\documentclass[12pt, a4paper]{amsart}

\usepackage[hmargin=30mm, vmargin=25mm, includefoot, twoside]{geometry}
\usepackage[bookmarksopen=true]{hyperref}

\usepackage{amscd}
\usepackage{amsfonts,amssymb,verbatim}
\usepackage{latexsym}
\usepackage{mathrsfs}
\usepackage{stmaryrd}
\usepackage{xspace}
\usepackage{enumerate, paralist}
\usepackage{graphicx}
\usepackage[all]{xy}
\usepackage{extarrows}

\usepackage[usenames,dvipsnames]{color}

\usepackage{txfonts, pxfonts}

\usepackage{amsthm}
\usepackage{amsmath}

\usepackage{todonotes}
\newtheorem{thm}{Theorem}[section]
 \newtheorem{cor}[thm]{Corollary}
 \newtheorem{lem}[thm]{Lemma}

\numberwithin{equation}{section}

 \theoremstyle{definition}
  \newtheorem{defn}[thm]{Definition}
  \newtheorem{question}[thm]{Question}

 \theoremstyle{remark}
 \newtheorem{rem}[thm]{Remark}
  \newtheorem{ex}[thm]{Example}

\newtheorem*{claim*}{Claim}

\def\NN{\mathbb{N}}
\def\RR{\mathbb{R}}

\def\ZZ{\mathbb{Z}}

\def\Nd{\mathcal{N}}

\def\max{\mathrm{max}}
\def\Im{\mathrm{Im}}

\begin{document}

\title[Examples of hyperbolic spaces]{Examples of hyperbolic spaces without the properties of quasi-ball or bounded eccentricity}

\author{Qizheng You and Jiawen Zhang}

\address[Q. You]{School of Mathematical Sciences, Fudan University, 220 Handan Road, Shanghai, 200433, China.}
\email{18307110161@fudan.edu.cn}

\address[J. Zhang]{School of Mathematical Sciences, Fudan University, 220 Handan Road, Shanghai, 200433, China.}
\email{jiawenzhang@fudan.edu.cn}

\date{}

\thanks{}

\begin{abstract}
In this note, we present examples of non-quasi-geodesic metric spaces which are hyperbolic (\emph{i.e.}, satisfying the Gromov's $4$-point condition) while the intersection of any two metric balls therein does not either ``look like'' a ball or has uniformly bounded eccentricity. This answers an open question posed in \cite{Ind07}.
\end{abstract}

\date{\today}
\maketitle

\parskip 4pt



\section{Introduction}\label{sec:intro}

In the seminal work \cite{Gro87}, Gromov introduced a notion of hyperbolicity for metric spaces which encodes the information of metric curvatures for the underlying spaces, with prototypes from classic hyperbolic geometry. Gromov's hyperbolic spaces have attractd a lot of interest since they are discovered and have fruitful applications in various aspects of mathematics (see, \emph{e.g.}, \cite{Mar99, GH90}).

Recall that a geodesic metric space $(X,d)$ is called \emph{hyperbolic} (in the sense of Gromov \cite{Gro87}) if there exists $\delta>0$ such that for any geodesic triangle in $(X,d)$, the union of the $\delta$-neighbourhoods of any two sides of the triangle contains the third. Gromov also provided a characterisation for his hyperbolicity using the so-called \emph{Gromov product}: 
\[
(x|y)_p = \frac{1}{2}(d(x,p)+d(y,p)-d(x,y)) \quad \text{for} \quad x,y,p\in (X,d).
\]
He proved in \cite{Gro87} that a geodesic metric space $(X,d)$ is hyperbolic \emph{if and only if} the following condition holds:

\noindent\textbf{Gromov's $4$-point condition:} There exists $\delta>0$ such that 
\[
(x|y)_p \geq \min \{(x|z)_p,(y|z)_p\} -\delta \quad \text{for all} \quad x,y,z,p\in (X,d).
\] 

Note that the statement of the Gromov's $4$-point condition does not require that the underlying space $(X,d)$ is geodesic. Hence in the general context, we say that a (not necessarily geodesic) metric space is \emph{hyperbolic} if the Gromov's $4$-point condition holds.

Later in \cite{Ind07}, Chatterji and Niblo discovered new characterisations of Gromov's hyperbolicity for geodesic metric spaces using the geometry of intersections of balls. More precisely, for a geodesic metric space, they showed that it is hyperbolic in the sense of Gromov \emph{if and only if} the following holds:

\noindent\textbf{Quasi-ball property:} The intersection of any two metric balls is at uniformly bounded Hausdorff distance from a ball.

They also considered the eccentricity of the intersection of balls. Recall that for a metric space $(X,d)$ and $\delta>0$, we say that the \emph{eccentricity} of a subset $S$ of $X$ is less than $\delta$ if there exist $c,c' \in X$ and $R\geq 0$ such that
  \begin{equation}\label{EQ:ecc}
    B(c,R) \subseteq S \subseteq B(c',R+\delta).
  \end{equation}
Here we use $B(x,r):=\{y\in X: d(x,y) \leq r\}$ to denote the metric ball. The \emph{eccentricity} of $S$ is the infimum of $\delta$ satisfying (\ref{EQ:ecc}). By convention, the eccentricity of the empty set is $0$. Chatterji and Niblo proved in \cite{Ind07} that a geodesic metric space is hyperbolic \emph{if and only if} the following holds:

\noindent\textbf{Bounded eccentricity property:} The intersection of any two metric balls has uniformly bounded eccentricity.

In \cite[Section 4]{Ind07}, Chatterji and Niblo also discussed the situation of non-geodesic metric spaces. They recorded an example due to Viktor Schroeder (see \cite[Example 18]{Ind07}) that there exists a non-geodesic metric space with the quasi-ball property but not hyperbolic (\emph{i.e.}, does not satisfy the Gromov's $4$-point condition). However, the other direction is unclear and hence, they asked the following:

\begin{question}\label{ques}
Does there exist a non-geodesic hyperbolic metric space which does not satisfy the quasi-ball property or the bounded eccentricity property?
\end{question}

In this short note, we provide an affirmative answer to Question \ref{ques} by constructing concrete examples. The main result is the following:

\begin{thm}\label{thm:main result}
There exists a non-quasi-geodesic hyperbolic (\emph{i.e.}, satisfying the Gromov's $4$-point condition) space which does not satisfy either the quasi-ball property or the bounded eccentricity property.
\end{thm}

Our construction is motivated by Gromov's observation in \cite[Section 1.2]{Gro87} (also suggested in \cite[Section 4]{Ind07}) that for a metric space $(X,d)$, we can endow another metric $d'$ on $X$ defined by 
\begin{equation}\label{EQ:new metric}
d'(x,y)=\ln (1+d(x,y)) \quad \text{for} \quad x,y\in X
\end{equation}
such that $(X,d')$ satisfies the Gromov's $4$-point condition. We show that if $(X,d)$ is geodesic and unbounded, then $(X,d')$ cannot be quasi-geodesic (see Corollary \ref{cor:non-quasi-geodesic}). Then we study the relation of the quasi-ball property and the bounded eccentricity property between $(X,d)$ and $(X,d')$ (see Lemma \ref{lem:quasi-ball property}, Lemma \ref{lem:ecc control} and Lemma \ref{negative prop}). Finally, we show in Example \ref{ex:ecc unbdd} that for the Euclidean space $X=\RR^2$ with the Euclidean metric $d$, the construction $(X,d')$ in (\ref{EQ:new metric}) provides an example to conclude Theorem \ref{thm:main result}.

\section{Preliminaries}\label{sec:pre}

Here we collect some necessary notions and notation for this note.

Let $(X,d)$ be a metric space. For $x \in X$ and $R \geq 0$, denote the (metric) ball by $B(x,R) = \{ y\in X : d(x,y) \leq R\}$. We say that $(X,d)$ is \emph{bounded} if there exist $x\in X$ and $R \geq 0$ such that $X = B(x,R)$, and \emph{unbounded} if it is not bounded. For a subset $A \subset X$ and $R\geq 0$, denote $\Nd_R(A) = \{x\in X: d(x,A) \leq R\}$ the $R$-\emph{neighbourhood} of $A$ in $X$. For subsets $A,B \subset X$, the \emph{Hausdorff distance} between $A$ and $B$ is 
\[
d_H(A,B) = \inf \{R \geq 0: A \subseteq \Nd_R(B) \text{ and } B \subseteq \Nd_R(A)\}.
\]


Recall that a \emph{path} in a metric space $(X,d)$ is a continuous map $\gamma: [a,b] \to X$. A path $\gamma$ is called \emph{rectifiable} if its \emph{length}
\[
\ell(\gamma):= \sup \left\{ \sum_{i=1}^{n} d(\gamma(t_{i-1}),\gamma(t_i)) : a=t_0< t_1 <\cdots < t_n=b , n\in \NN \right\}
\]
is finite. Usually it is convenient to change the parameter $t\in [a,b]$ to the \emph{standard arc parameter} $s\in [0,\ell(\gamma)]$ as follows. Define a map $\varphi:[a,b] \rightarrow [0,\ell(\gamma)]$ by $t \mapsto s=\ell(\gamma|_{[a,t]})$, and set $\tilde{\gamma}:[0,\ell(\gamma)] \to X$ by $\tilde{\gamma}:=\gamma \circ \varphi^{-1}$. Then we have $\ell(\tilde{\gamma}|_{[s_1, s_2]}) = |s_1 - s_2|$ for any $0 \leq s_1 \leq s_2 \leq \ell(\gamma)$.

Now we recall the notion of (quasi-)geodesics.

\begin{defn}\label{defn:(quasi-)geodesic}
  Let $(X,d)$ be a metric space.
  \begin{enumerate}
    \item Given $x,y \in X$, a \emph{geodesic} between $x$ and $y$ is an isometric embedding $\gamma : [0,d(x,y)] \rightarrow X$ with $\gamma(0)=x$ and $\gamma(d(x,y))=y$.
    The space $(X,d)$ is called \emph{geodesic} if for any $x,y$ in $X$, there exists a geodesic between $x$ and $y$.
    \item Given $x,y \in X, L \geq 1$ and $C \geq 0$, an \emph{$(L,C)$-quasi-geodesic} between $x$ and $y$ is a map $\gamma: [0,T] \rightarrow X$ such that $\gamma(0)=x, \gamma(T)=y$ and
    \begin{equation*}
      \frac{1}{L}|a-b|-C \leq d(\gamma(a),\gamma(b)) \leq L|a-b|+C \quad \text{for all} \quad a,b\in [0,T].
    \end{equation*}
    The space $(X,d)$ is called \emph{$(L,C)$-quasi-geodesic} if for any $x,y$ in $X$, there exists an $(L,C)$-quasi-geodesic between $x$ and $y$. We also say that $(X,d)$ is \emph{quasi-geodesic} if it is $(L,C)$-quasi-geodesic for some $L$ and $C$.
  \end{enumerate}
\end{defn}

We also need the notion of ultrametric space. Recall that a metric space $(X,d)$ is called \emph{ultrametric} if there exists $\delta>0$ such that for any points $x,y,z \in X$, we have
    \begin{equation*}
      d(x,y) \leq \max \{d(x,z),d(y,z)\} + \delta.
    \end{equation*}
The following is due to Gromov:

\begin{lem}[{\cite[Section 1.2]{Gro87}}]\label{lem:ultrametric is hyperbolic}
An ultrametric space satisfies the Gromov's $4$-point condition.
\end{lem}

Recall from Section \ref{sec:intro} that for a metric space $(X,d)$, Gromov considered another metric $d'$ on $X$ defined in (\ref{EQ:new metric}) and noticed that
\begin{eqnarray*}
  d'(x,y) &\leq& \ln (1+d(x,z)+d(z,y)) \\
  &\leq& \ln(2+2\max \{ d(x,z),d(y,z) \}) \\
  &=& \max \{ d'(x,z),d'(y,z) \} + \ln 2.
\end{eqnarray*}
Combining with Lemma \ref{lem:ultrametric is hyperbolic}, we obtain the following:

\begin{lem}[{\cite[Section 1.2]{Gro87}}]\label{lem:new metric is hyperbolic}
For a metric space $(X,d)$, the new metric $d'$ on $X$ defined in (\ref{EQ:new metric}) satisfies the Gromov's $4$-point condition.
\end{lem}

\section{Proof of Theorem \ref{thm:main result}}

%

This whole section is devoted to the proof of Theorem \ref{thm:main result}, which is divided into several parts.

Firstly, we would like to study the property of (quasi-)geodesic for the new metric $d'$ defined in (\ref{EQ:new metric}). To simplify the notation, for a path $\gamma$ in $X$, we denote $\ell(\gamma)$ and $\ell'(\gamma)$ its length with respect to the metric $d$ and $d'$, respectively.

We need the following lemma:

\begin{lem}\label{length}
  Let $(X,d)$ be a metric space and $d'$ be the metric on $X$ defined in (\ref{EQ:new metric}).
A path $\gamma : [a,b] \rightarrow X$ in $X$ is rectifiable with respect to $d$ \emph{if and only if} it is rectifiable with respect to $d'$. In this case, we have $\ell(\gamma) = \ell'(\gamma)$.
\end{lem}

\begin{proof}
Firstly, we assume that $\gamma$ is rectifiable with respect to $d$, \emph{i.e.}, $\ell(\gamma)<\infty$. Note that $\ln(1+x) \leq x$ holds for all $x \geq 0$. Hence for any partition $a=t_0< t_1 <\cdots < t_n=b$, we have
\[
\sum_{i=1}^{n} d'(\gamma(t_{i-1}),\gamma(t_i)) \leq \sum_{i=1}^{n} d(\gamma(t_{i-1}),\gamma(t_i)) \leq \ell(\gamma),
\]
which implies that $\ell'(\gamma) \leq \ell(\gamma)$. In particular, $\gamma$ is rectifiable with respect to $d'$.


Conversely, we assume that $\gamma$ is rectifiable with respect to $d'$, \emph{i.e.}, $\ell'(\gamma)<\infty$. Without loss of generality, we can assume that $\gamma$ is parametrised by the standard arc parameter with $a=0$ and $b=\ell'(\gamma)$.
Note that for any $\alpha \in (0,1)$, there exists $\delta>0$ such that $\alpha x \leq \ln (1+x)$ holds for all $x \in [0,\delta]$. 
Given a partition $a=s_0< s_1 <\cdots < s_m=b$, we choose a refinement $a=t_0< t_1 <\cdots < t_n=b$ such that $|t_{i-1}-t_i|<\ln(1+\delta)$ holds for all $i$. Hence we have
  \begin{equation*}
    d'(\gamma(t_{i-1}),\gamma(t_i)) \leq l'(\gamma \vert_{[t_{i-1},t_i]})  =|t_{i-1}-t_i| < \ln (1+\delta),
  \end{equation*}
 which implies that $d(\gamma(t_{i-1}),\gamma(t_i)) < \delta$ due to (\ref{EQ:new metric}).
 Therefore, we obtain
 \[
 \sum_{i=1}^{m} d(\gamma(s_{i-1}),\gamma(s_i)) \leq \sum_{i=1}^{n} d(\gamma(t_{i-1}),\gamma(t_i)) \leq \frac{1}{\alpha} \cdot \sum_{i=1}^{n} d'(\gamma(t_{i-1}),\gamma(t_i)) \leq \frac{1}{\alpha} \cdot \ell'(\gamma)
 \]
for all $\alpha \in (0,1)$. Letting $\alpha \to 1$ and taking the supremum of the left hand side, we obtain that $\ell(\gamma) \leq \ell'(\gamma)$, which concludes the proof.
\end{proof}

As a direct corollary, we obtain the following:

\begin{cor}\label{cor:non-geodesic}
Let $(X,d)$ be a geodesic metric space which contains at least two elements, and $d'$ be the metric defined in (\ref{EQ:new metric}). Then the metric space $(X,d')$ is \emph{not} geodesic.
\end{cor}

\begin{proof}
By assumption, we take two distinct points $x, y \in X$. If $(X,d')$ is geodesic, we choose a geodesic $\gamma$ between $x$ and $y$. In particular, $\gamma$ is rectifiable with respect to $d'$. Hence by Lemma \ref{length}, we know that $\gamma$ is also rectifiable with respect to $d$ and we have
\[
 \ell'(\gamma) = \ell(\gamma) \geq d(x,y) > d'(x,y),
\]
where the last inequality follows from the assumption that $x\neq y$.
This is a contradiction to the assumption that $\gamma$ is a geodesic between $x$ and $y$ with respect to the metric $d'$. Hence we conclude the proof.
\end{proof}

Moreover, with an extra hypothesis, the new metric $d'$ cannot be even quasi-geodesic.

\begin{cor}\label{cor:non-quasi-geodesic}
Let $(X,d)$ be an unbounded geodesic metric space, and $d'$ be the metric defined in (\ref{EQ:new metric}). 
Then for any $L \geq 1$ and $C \geq 0$, the metric space $(X,d')$ cannot be $(L,C)$-quasi-geodesic.
\end{cor}

To prove Corollary \ref{cor:non-quasi-geodesic}, we need the following lemma to tame quasi-geodesics. The idea is similar to \cite[Lemma III.H.1.11]{Mar99}, but the setting is slightly different.



\begin{lem}\label{tame quasi-geodesics}
Let $(X,d)$ be a geodesic metric space and $d'$ be the metric defined in (\ref{EQ:new metric}).
  Given an $(L,C)$-quasi-geodesic $\gamma:[a,b]\rightarrow X$ with respect to $d'$, there exists a continuous and rectifiable $(L,C')$-quasi-geodesic $\gamma': [a,b] \rightarrow X$ with respect to $d'$ satisfying the following:
  \begin{enumerate}
    \item $\gamma'(a)=\gamma(a)$ and $\gamma'(b)=\gamma(b)$;
    \item $C'=3(L+C)$;
    \item $\ell'(\gamma'|_{[t,t']}) \leq k_1 d'(\gamma'(t),\gamma'(t')) + k_2$ for all $t,t' \in [a,b]$, where $k_1=L(L+C)$ and $k_2 =(LC'+4)(L+C)$;
    \item $d'_H(\Im(\gamma), \Im(\gamma')) \leq L+C$.
  \end{enumerate}
\end{lem}

Careful readers might already notice that in the situation of \cite[Lemma III.H.1.11]{Mar99}, we need to assume that the \emph{new} metric $d'$ on $X$ is geodesic instead of the current setting that the \emph{original} metric $d$ is geodesic. Although the proof is similar, here we also provide one for convenience to readers.

\begin{proof}[Proof of Lemma \ref{tame quasi-geodesics}]
Define $\gamma'$ to agree with $\gamma$ on $\Sigma := \{a,b\} \cup (\ZZ \cap (a,b))$, then choose geodesic segments with respect to $d$ joining the images of successive points in $\Sigma$ and define $\gamma'$ by concatenating linear reparameterisations of these geodesic segments.

  Let $[t]$ denote the point of $\Sigma$ closest to $t \in [a,b]$. Note that the $d'$-distance of the images of successive points in $\Sigma$ is at most $L+C$, and hence we have
  \[
    d'(\gamma'(t),\gamma'([t])) = \ln \left(1+ d(\gamma'(t),\gamma'([t]))\right) \leq \ln (1+\exp(L+C)-1) \leq L+C,
  \] 
  which implies that $d'_H(\Im(\gamma), \Im(\gamma')) \leq L+C$. 
  Since $\gamma$ is an $(L,C)$-quasi-geodesic with respect to $d'$, and $\gamma([t]) = \gamma'([t])$ for all $t\in [a,b]$, we have
  \[ 
    d'(\gamma'(t),\gamma'(t')) \leq d'(\gamma'([t]),\gamma'([t'])) + 2(L+C) \leq L|[t]-[t']| + C +2(L+C) \leq L|t-t'| +3(L+C),
  \]
  and similarly, we have 
\begin{equation}\label{EQ:est}
    d'(\gamma'(t),\gamma'(t')) \geq \frac{1}{L}|t-t'| -3(L+C)
\end{equation}
  for all $t,t' \in [a,b]$. Hence $\gamma'$ is an $(L,C')$-quasi-geodesic with respect to $d'$.

  For any integers $s,s' \in \Sigma$ with $s \leq s'$, Lemma \ref{length} tells us that
  \[
    \ell'(\gamma' \vert _{[s,s']}) = \ell (\gamma' \vert _{[s,s']}) = \sum_{k=s}^{s'-1} \ell (\gamma' \vert _{[k,k+1]}) \leq |s-s'|(L+C).
  \]
  Similarly for any $s,s' \in \Sigma$, we have $\ell'(\gamma' \vert _{[s,s']}) \leq (|s-s'|+2)(L+C)$. Hence for any $t,t' \in [a,b]$, we have
  \[
    \ell'(\gamma' \vert_{[t,t']}) \leq (|[t]-[t']|+2)(L+C) + (L+C) \leq (|t-t'|+4)(L+C).
  \]
  Combining with Inequality (\ref{EQ:est}),
  we obtain that $\ell(\gamma'|_{[t,t']}) \leq k_1 d'(\gamma'(t),\gamma'(t')) + k_2$ for $k_1,k_2$ defined in (3). Hence we conclude the proof.
\end{proof}


\begin{proof}[Proof of Corollary \ref{cor:non-quasi-geodesic}]
Assume that $(X,d')$ is $(L,C)$-quasi-geodesic for some $L \geq 1$ and $C \geq 0$. For any $x,y\in X$, choose an $(L,C)$-quasi geodesic $\gamma:[0,T] \rightarrow X$ (with respect to $d'$) connecting them. Lemma \ref{tame quasi-geodesics} implies that there is a rectifiable path $\gamma':[0,T] \rightarrow X$ which is an $(L,3L+3C)$-quasi-geodesic (with respect to $d'$) connecting $x$ and $y$. Moreover, we have $\ell'(\gamma') \leq k_1 d'(x,y) + k_2$ for $k_1=L(L+C)$ and $k_2 =(3L(L+C)+4)(L+C)$. Hence Lemma \ref{length} implies that 
\[
    k_1 d'(x,y) + k_2 \geq \ell'(\gamma') = \ell(\gamma') \geq d(x,y) = \exp(d'(x,y)) -1.
\]
Note that $(X,d)$ is unbounded, which implies that $(X,d')$ is also unbounded. Hence taking $d'(x,y) \to \infty$, we conclude a contradiction and finish the proof.
%
\end{proof}

Next, we move to study the relation of the quasi-ball property and the bounded eccentricity property between the metric spaces $(X,d)$ and $(X,d')$. Again to save the notation, for $x\in X$ and $r\geq 0$, we denote $B(x,r)$ and $B'(x,r)$ the closed balls with respect to the metrics $d$ and $d'$, respectively. For $A \subset X$ and $\delta \geq 0$, we denote $\Nd_\delta(A)$ and $\Nd'_\delta(A)$ the $\delta$-neighbourhood of $A$ with respect to the metrics $d$ and $d'$, respectively.



\begin{lem}\label{lem:quasi-ball property}
Let $(X,d)$ be a metric space and $d'$ be the metric on $X$ defined in (\ref{EQ:new metric}). Then $(X,d)$ has the quasi-ball property \emph{if and only if} $(X,d')$ has the quasi-ball property.
\end{lem}

\begin{proof}
Firstly, we assume that $(X,d')$ has the quasi-ball property, \emph{i.e.}, there exists $\delta \geq 0$ such that the intersection of any two balls in $(X,d')$ is $\delta$-close to another ball (\emph{i.e.}, their Hausdorff distance is bounded by $\delta$). Note that there exists $\alpha=\alpha(\delta) >1$ such that $x \leq \alpha \ln(1+x)$ for all $x \in [0,\delta]$. This implies that for $A,B \subseteq X$ with $A \subseteq \Nd'_\delta(B)$, we have $A \subseteq \Nd_{\alpha\delta}(B)$.

Given two balls $B(x,s)$ and $B(y,t)$ in $(X,d)$, it is clear that $B(x,s) = B'(x,\ln(1+s))$ and $B(y,t) = B'(y,\ln(1+t))$. Hence by the assumption, there exists another ball $B'(c,r)$ in $(X,d')$ such that 
\[
d'_H\left( B'(x,\ln(1+s)) \cap B'(y,\ln(1+t)), B'(c,r) \right)  \leq \delta.
\]
It follows from the previous paragraph that in this case, we have
\[
d_H\left( B(x,s) \cap B(y,t), B(c, \exp(r)-1) \right) = d_H\left( B'(x,\ln(1+s)) \cap B'(y,\ln(1+t)), B'(c,r) \right) \leq \alpha \delta.
\]
Hence $(X,d)$ has the quasi-ball property. 

The converse holds similarly, using the fact that $\ln(1+x) \leq x$ holds for all $x \geq 0$.
\end{proof}

Combining Lemma \ref{lem:new metric is hyperbolic}, Corollary \ref{cor:non-quasi-geodesic} and Lemma \ref{lem:quasi-ball property}, we obtain the following, which concludes part of Theorem \ref{thm:main result}.

\begin{cor}\label{cor:quasi-ball ceg}
Let $(X,d)$ be a geodesic metric space which is not hyperbolic (\emph{e.g.}, the Euclidean space $\RR^2$ with the Euclidean metric) and $d'$ be the metric on $X$ defined in (\ref{EQ:new metric}). Then $(X,d')$ is a non-quasi-geodesic metric space which satisfies the Gromov's $4$-point condition but \emph{not} the quasi-ball property.
\end{cor}


Concerning the bounded eccentricity property, we have the following:

\begin{lem}\label{lem:ecc control}
Let $(X,d)$ be a metric space and $d'$ be the metric on $X$ defined in (\ref{EQ:new metric}). If $(X,d)$ satisfies the bounded eccentricity property, so does $(X,d')$.
\end{lem}

\begin{proof}
Given two balls $B'(x,\ln(1+r))$ and $B'(y,\ln(1+s))$ in $(X,d')$, we assume that their intersection $Y$ is non-empty. Note that 
\[
Y=B'(x,\ln(1+r)) \cap B'(y,\ln(1+s)) = B(x,r)\cap B(y,s)
\]
is again an intersection of balls in $(X,d)$. Hence by the assumption, there exist $c,c' \in X$ and $R \geq 0$ such that
  \begin{equation*}
    B(c,R) \subseteq Y \subseteq B(c',R+\delta_0).
  \end{equation*}
  Therefore in $(X,d')$, we have
  \begin{equation*}
    B'(c,\ln (1+R)) \subseteq Y \subseteq B'(c',\ln (1+R+\delta_0)).
  \end{equation*}
  Then the eccentricity of $Y$ in $(X,d')$ is bounded above by
  \begin{equation*}
    \ln (1+R+\delta_0) - \ln (1+R) = \ln (1+\frac{\delta_0}{1+R}) \leq \ln (1+\delta_0) \leq \delta_0,
  \end{equation*}
  which concludes the proof.
\end{proof}

\begin{rem}\label{rem:ecc converse control}
Readers might wonder whether the converse to Lemma \ref{lem:ecc control} holds. Unfortunately, in general this is false. Note that if we have
\begin{equation*}
  B'(c,\ln (1+R)) \subseteq Y \subseteq B'(c',\ln (1+R) + \delta_0)
\end{equation*}
in $(X,d')$, then it implies that
\begin{equation*}
  B(c,R) \subseteq Y \subseteq B(c',\exp(\delta_0) \cdot (1+R)-1)
\end{equation*}
in $(X,d)$. Hence the eccentricity of $Y$ in $(X,d)$ is bounded by $(\exp(\delta_0)-1)(1+R)$. However, if $R$ (\emph{i.e.}, the radius of the ball contained in $Y$) does not have a uniform upper bound, then neither does the eccentricity of $Y$.
\end{rem}


Remark \ref{rem:ecc converse control} suggests the following partial converse to Lemma \ref{lem:ecc control}:

\begin{lem}\label{negative prop}
Let $(X,d)$ be a metric space and $d'$ be the metric on $X$ defined in (\ref{EQ:new metric}). Assume that there exists a sequence of subsets $\{Y_n\}_{n\in \NN}$ of $X$ satisfying the following:
  \begin{enumerate}
    \item each $Y_n$ is the intersection of two balls in $(X,d)$;
    \item there exists $M >0$ such that for all $n \in \NN$, the radius (with respect to $d$) of any ball contained in $Y_n$ is bounded above by $M$;
    \item the eccentricity of $Y_n$ in $(X,d)$ is not uniformly bounded.
  \end{enumerate}
  Then $(X,d')$ does \emph{not} satisfy the bounded eccentricity property.
\end{lem}

\begin{proof}
Assume the opposite, \emph{i.e.}, there exists $\delta_0>0$ such that the eccentricity of the intersection of any two balls in $(X,d')$ is uniformly bounded by $\delta_0$. Hence by condition (1) we know that for each $n\in \NN$, there exist $c_n,c_n' \in X$ and $r_n \geq 0$ such that
  \begin{equation*}
    B'(c_n,\ln (1+r_n)) \subseteq Y_n \subseteq B'(c_n', \ln (1+r_n)+\delta_0).
  \end{equation*}
Hence in $(X,d)$, we have
\begin{equation*}
    B(c_n,r_n) \subseteq Y_n \subseteq B(c_n,\exp(\delta_0) \cdot (1+r_n)-1).
  \end{equation*}
By condition (2), we know that $r_n \leq M$ for all $n\in \NN$. Therefore, we have
\begin{equation*}
    \exp(\delta_0) \cdot (1+r_n)-1-r_n =  (\exp(\delta_0)-1)(1+r_n) \leq (\exp(\delta_0)-1)(1+M),
  \end{equation*}
which is a contradiction to condition (3) in the assumption.
\end{proof}

\begin{ex}\label{ex:ecc unbdd}
Now we show that there exists a non-quasi-geodesic metric space which satisfies the Gromov's $4$-point condition but \emph{not} the bounded eccentricity property. For example, take $(X,d)$ to be the Euclidean space $\RR^2$ equipped with the Euclidean metric, and let $d'$ be the metric on $X$ defined in (\ref{EQ:new metric}). Lemma \ref{lem:new metric is hyperbolic} and Corollary \ref{cor:non-quasi-geodesic} imply that $(X,d')$ is not quasi-geodesic but satisfies the Gromov's $4$-point condition. 

For each $n\in \NN$, take $x_n=(0,0)$ and $y_n=(2n,0)$ and set
\[
Y_n = B(x_n,n+1) \cap B(y_n,n+1).
\]
It is easy to see that in $(X,d)$, the biggest ball contained in $Y_n$ is $B((n,0),1)$. Moreover, the diameter of $Y_n$ is
\[
d\left((n,\sqrt{2n+1}),(n,-\sqrt{2n+1})\right) = 2\sqrt{2n+1},
\]
which implies that the eccentricity of $Y_n$ cannot be uniformly bounded. Therefore applying Lemma \ref{negative prop}, we conclude the result.
\end{ex}

\begin{proof}[Proof of Theorem \ref{thm:main result}]
Combining Corollary \ref{cor:quasi-ball ceg} and Example \ref{ex:ecc unbdd}, we conclude the proof for Theorem \ref{thm:main result}.
\end{proof}

Finally, recall that in \cite{Wen08} the following weaker form of the bounded eccentricity property was considered: For a metric space $(X,d)$, there exist $\lambda>0$ and $\delta>0$ such that for any two balls $B_1, B_2 \subseteq X$ with non-empty intersection there exist $z,z' \in X$ and $r \geq 0$ such that
\[
B(z,r) \subseteq B_1 \cap B_2 \subseteq B(z', \lambda r+ \delta).
\]
Wenger showed in \cite{Wen08} that this condition also implies Gromov's hyperbolicity for geodesic metric spaces. 

We remark that this weaker form of the bounded eccentricity property does not hold either for the space $(X,d')$ constructed in Example \ref{ex:ecc unbdd}.

\bibliographystyle{plain}
\bibliography{hyperbolicity}

\end{document}